\documentclass[12pt]{article}

\usepackage{amsfonts, amsmath, amssymb, amsthm}

\usepackage{hyperref}

\usepackage{a4}

\usepackage{blkarray}

\newtheorem{theorem}{Theorem}

\theoremstyle{remark}
\newtheorem*{remark}{Remark}

\title{Heptagon relations parameterized by simplicial 3-cocycles}
\author{Igor G. Korepanov}

\date{March 2021}

\begin{document}

\sloppy

\maketitle

\begin{abstract}
We construct a family of heptagon relations---algebraic imitations of five-dimen\-sional Pachner move 3--4, parameterized by simplicial 3-cocycles.
\end{abstract}

\section{Introduction}\label{s:i}

\subsection{Five-dimen\-sional Pachner moves and heptagon relations: generalities}\label{ss:g}

A piecewise linear (PL) manifold triangulation can be transformed into any other triangulation by means of a sequence of \emph{Pachner moves}~\cite{Pachner,Lickorish}. Algebraic imitations of such moves are often called \emph{polygon relations}~\cite{DM-H}; they can be used in constructing discrete topological field theories which can further be applied to both topological and mathematical physical problems.

In particular, in five dimensions there are six types of Pachner moves: 1--6, 2--5, 3--4 and their inverses. Notation ``$m$--$n$'' means that we take a cluster of $m$ simplices of the highest dimension---five in this case, which are thought of as making part of a triangulation, and replace it with a cluster of $n$ simplices; this is feasible because the two clusters have the same boundary. We call the initial cluster of $m$ simplices the \emph{left-hand side} (l.h.s.) of the Pachner move, while its resulting cluster of $n$ simplices---its \emph{right-hand side} (r.h.s.).

A 5-simplex has six faces---4-cells; suppose we have a set~$X$, called set of \emph{colors}, with which we can color these faces. All possible colorings of a 5-simplex---we denote it $\Delta^5$ or often simply~$v$---belong thus to~$X^{\times 6}$---the sixth Cartesian degree of~$X$; but we postulate that not all colorings are \emph{permitted}: there is a subset $R_v \subset X^{\times 6}$ of permitted colorings, depending moreover on~$v$.

When we glue simplices together to obtain a cluster~$M$, the condition that each simplex must be colored permittedly results in the fact that not all colorings may appear on its boundary~$\partial M$, but only some subset $R_{\partial M} \subset X^{\times N_4}$, where $N_4$ is the number of 4-cells in~$\partial M$. In this paper, \emph{hexagon relation} means, by definition, that the two subsets~$R_{\partial M}$ coincide for the l.h.s.\ and r.h.s.\ of a Pachner move; as we already indicated, the boundary is the same for both.

Specifically, we study in this paper Pachner move 3--4, believed to be ``central'' or at least good to start with. If we describe simplices in terms of their vertices, and call these latter $1,\ldots,7$, then the l.h.s.\ of this move consists of 5-simplices 123456, 123457 and 123467, while the r.h.s.---of simplices 123567, 124567, 134567 and 234567. Below, in Theorems \ref{th:3,9,18} and~\ref{th:9,18}, we will need the numbers $N_0=7$ of vertices, $N_1=21$ of edges, $N_2=34$ of triangles and $N_4=12$ of 4-simplices in their common boundary; it is an easy exercise to find these numbers: for instance, $N_2=34$ because, of $\binom{7}{3}=35$ triangles with vertices $1,\ldots,7$, only triangle~$567$ does not belong to the boundary.

\subsection{Specific features of higher polygon relations: linearity, parameterization, cohomology}

\paragraph{Linearity.}
One feature of known higher polygon relations: hexagon, heptagon, etc., is their \emph{linearity}. This means that the set of colors is a \emph{field}: $X=F$, or sometimes a ring, and the conditions specifying permitted colorings for a simplex (4-simplex in the hexagon case, 5-simplex for heptagon, etc.) are linear. Field~$F$ can be finite of course, but may be required not to be very small, so that some ``general position'' requirements can be realized within it.

\paragraph{Parameterization.}
Earlier, two different types of heptagon relations were found, see \cite{heptagon} and~\cite{DK}. There is, however, one important thing they have in common: both are parameterized by elements of a Grassmannian, and Grassmannians are built from quantities living in the seven \emph{vertices} (called ``$1,\ldots,7$'' in our Subsection~\ref{ss:g}).

At the same time, \emph{hexagon} relations are known to be naturally related to simplicial 2-cocycles~\cite{bosonic,nonconstant}, that is, quantities living on simplices of dimension~$>\nobreak 0$. It makes sense thus to search for higher polygon solutions related to higher cocycles. And this is what we do in the present paper: construct \emph{heptagon} relations parameterized by simplicial 3-cocycles.

\paragraph{Cohomology.}
Hexagon relations are known to have nontrivial cohomology---see, for instance,~\cite{nonconstant} and references therein. The same is expected from other higher polygon relations.

\subsection{Contents of the rest of this paper}

In Section~\ref{s:v}, we describe what we call ``V-colorings'' of a simplicial complex. They form a linear space spanned by ``triangle vectors'', where a triangle vector means a coloring with nonzero components only on the faces containing a given triangle.

Dually, in Section~\ref{s:f}, we describe what we call ``F-colorings''. They form a linear space \emph{annihilated} by ``triangle functionals'', these latter being linear forms depending only on colors on the faces containing a given triangle. We show that the linear space of F-colorings for one 5-simplex \emph{coincides} with the space of V-colorings.

In Section~\ref{s:h}, we introduce the linear space of ``permitted'' colorings. For one 5-simplex, permitted colorings coincide, by definition, with V- and F-colorings. For a bigger triangulation, each 5-simplex must be colored permittedly (as we already mentioned in Subsection~\ref{ss:g}). We analyze what happens for the l.h.s.\ and r.h.s.\ of Pachner move 3--4, and derive the heptagon relation.

\section{Triangle vectors and V-colorings}\label{s:v}

\subsection{Definitions}\label{ss:vd}

We consider simplicial complexes~$K$ with \emph{colored 4-cells}. In this paper, this will mean that each 4-cell---\emph{pentachoron}, typically denoted below as~$u$ or~$\Delta^4$---is assigned a three-component color---element of~$F^3$, where $F$ is a fixed field. By definition, the color changes its sign when pentachoron changes its orientation. A \emph{triangle vector}~$\mathsf e_s$ for a triangle $s=ijk$ is such a coloring that can have nonzero components only on faces $u\supset s$. By definition, $\mathsf e_s$ changes its sign when $s$ changes its orientation:
\begin{equation}\label{vo}
\mathsf e_{ijk} = -\mathsf e_{jik} = \mathsf e_{jki}\,, \quad \text{etc.}
\end{equation}

We define \emph{V-colorings} of~$K$ as elements of the linear space spanned by a given set of triangle vectors for each~$s$.

\begin{remark}
``V'' in ``V-colorings'' stands for triangle \textbf{v}ectors, to distinguish these colorings from ``F-colorings'' that will be introduced in Section~\ref{s:f} and will be determined by ``triangle \textbf{f}unctionals''.
\end{remark}

Below, we consider only such sets of triangle vectors that have the following properties:
\begin{itemize}
 \item[$\mathrm{(i)}$] \emph{vanishing boundary}: for any given oriented edge~$b$, 
  \begin{equation}\label{ve}
   \sum_{s\supset b} \epsilon _{s,b} \mathsf e_s =0,
  \end{equation}
  where $\epsilon _{s,b}$ is the sign with which $b$ enters in the boundary~$\partial s$ of~$s$,
 \item[$\mathrm{(ii)}$] \emph{orthogonality to a given 2-cochain}:
  \begin{equation}\label{vw}
   \sum_{\mathrm{all\;}s} \beta _s \mathsf e_s =0,
  \end{equation}
   where $\{\beta _s\}$ are generic coefficients for all triangles~$s$; for reasons that will become clear soon we will treat them together as a simplicial 2-\emph{cochain}~$\beta$,
 \item[$\mathrm{(iii)}$] \emph{nondegeneracy}: for any pentachoron~$u$, its colors---the restrictions~$\mathsf e_s|_u$ of triangle vectors on~$u$---span the whole~$F^3$.
\end{itemize}

\subsection{Existence and uniqueness up to gauge automorphisms}

\begin{theorem}\label{th:euv}
For a given~$K$, there exists a set of triangle vectors~$\mathsf e_s$ satisfying the above conditions $\mathrm{(i)}$--$\mathrm{(iii)}$. Moreover, all such~$\mathsf e_s$ are determined uniquely up to automorphisms of the (three-dimen\-sional) color spaces on each separate 4-face.
\end{theorem}

We will call these automorphisms ``gauge automorphisms''.

\begin{proof}
First, the statement of the theorem is clear if our simplicial complex~$K$ consists of just one pentachoron: $K=\Delta^4$. Indeed, there are 10 triangle vectors corresponding to $\binom{5}{3}=10$ triangles, on which there are
\begin{equation}\label{v-face}
\underset{\substack{{\vphantom{x}}\\\mathrm{number}\\ \mathrm{of\; edges}}}{10} - \underset{\substack{{\vphantom{x}}\\\mathrm{number}\\ \mathrm{of\; vertices}}}{5} + \underset{\substack{{\vphantom{x}}\\\mathrm{one}\\ \mathrm{constant}}}{1} + \underset{\substack{{\vphantom{x}}\\\mathrm{restriction}\\ (\ref{vw})}}{1} = 7
\end{equation}
restrictions. The first three terms in the l.h.s.\ of~\eqref{v-face} are due to the fact that there are relations~\eqref{ve} on the edges, which have of course, in their turn, linear dependencies between them associated with \emph{vertices}, while these latter have, again, one linear dependence between them.

For an arbitrary~$K$, to prove the existence, we simply compose~$\mathsf e_s$ from 3-vectors on each~$\Delta^4$, that is, take a direct sum of these latter. The uniqueness up to gauge automorphisms follows from the fact that \eqref{ve} and~\eqref{vw} imply the same conditions for the restrictions of~$\mathsf e_s$ on every~$\Delta^4$ as if~$K$ consisted of only this~$\Delta^4$.
\end{proof}

\subsection{Triangle vectors and a simplicial 3-cocycle}

The cases interesting for us in this paper are when $K$ is, if not just one pentachoron, then either the boundary~$\partial\Delta^5$ of one six-vertex simplex~$\Delta^5$, or the common boundary of the l.h.s.\ and r.h.s.\ of the Pachner move 3--4. These cases are what we have primarily in mind when formulating our next theorems.

\begin{theorem}\label{th:b'}
Let complex~$K$ represent a triangulation of a compact connected 4-dimen\-sional PL manifold, possibly with boundary, and having a trivial second cohomology~$H^2(K,F)=0$. Vectors~$\mathsf e_s$, satisfying conditions $\mathrm{(i)}$--$\mathrm{(iii)}$ of Subsection~\ref{ss:vd} and taken to within gauge automorphisms, depend only on the coboundary $\omega=\partial \beta$ of cochain~$\beta$, that is, numbers
\begin{equation}\label{omega}
\omega _{ijkl} = \beta _{ijk}-\beta _{ijl}+\beta _{ikl}-\beta _{ijl}
\end{equation}
attached to tetrahedra~$t=ijkl$.
\end{theorem}

\begin{remark}
Cochain~$\omega$, being a 3-coboundary, is of course also a simplicial 3-cocycle, and we often prefer using this latter word. ``Cocycle'' and ``coboundary'' make no difference in this paper, where we consider just a heptagon relation which is ``local'' from the viewpoint of a topological quantum field theory (TQFT). But ``cocycle'' may become more relevant when we pass, in a future work, to constructing an actual TQFT.
\end{remark}

\begin{proof}
Indeed, if there is another 2-cochain~$\beta'$ such that $\partial \beta' = \partial \beta$, then $\beta'-\beta$ is a cocycle and hence a coboundary, and the difference
\[
\sum_{\mathrm{all}\;s} \beta'_s \mathsf e_s - \sum_{\mathrm{all}\;s} \beta _s \mathsf e_s
\]
of the corresponding left-hand sides of~\eqref{vw} can be represented as a linear combination of the left-hand sides of relations~\eqref{ve}. Which means that \eqref{ve} and~\eqref{vw} together give the same for $\beta$ and~$\beta'$.
\end{proof}

\subsection{Dimensions of linear spaces of V-colorings}

\begin{theorem}\label{th:vd}
Let $K$ be as in Theorem~\ref{th:b'}, and let $N_i$ denote the number of $i$-dimen\-sional simplices in it. Then, the dimension of the space of its V-colorings is
\begin{equation}\label{vM}
N_2-N_1+N_0-2.
\end{equation}
\end{theorem}

\begin{remark}
This is of course in perfect agreement with~\eqref{v-face}: if we subtract the l.h.s.\ of~\eqref{v-face} from $N_2=10$, we obtain the relevant particular case of~\eqref{vM}.
\end{remark}

\begin{proof}
The situation here is as follows: we have $N_2$ triangle vectors and linear relations \eqref{ve} and~\eqref{vw} between them; of these latter there are, clearly, $N_1-N_0+2$ independent. The point is to show that there are no more independent relations.

Indeed, suppose there is a relation
\begin{equation}\label{g-hm}
\sum_{\mathrm{all}\;s} \kappa _s \mathsf e_s = 0.
\end{equation}
Projecting~\eqref{g-hm} on any chosen pentachoron~$u$ and reasoning like we did in the proof of Theorem~\ref{th:b'}, we find that the coboundary~$\partial\kappa$ must be proportional to $\omega$~\eqref{omega} on~$u$. But, as we remember, our complex~$K$ represents a connected 4-manifold, wherefrom we deduce at once that $\partial\kappa$ is proportional to $\omega$ everywhere, hence, \eqref{g-hm} brings about no new restrictions.
\end{proof}

As a direct application of Theorem~\ref{th:vd}, we obtain the following result for the dimensions of specific important spaces of V-colorings.

\begin{theorem}\label{th:3,9,18}
Conditions $\mathrm{(i)}$--$\mathrm{(iii)}$ of Subsection~\ref{ss:vd} give 
 \begin{itemize}
  \item a 3-dimensional space of V-colorings for a separate~$\Delta^4$,
  \item a 9-dimensional space of V-colorings for the boundary of one 5-simplex~$\Delta^5$,
  \item a 18-dimensional space of V-colorings for the common boundary of the l.h.s.\ and r.h.s.\ of the Pachner move 3--4.
 \end{itemize}
\end{theorem}

\begin{proof}
Nothing of course remains to be proven for a separate~$\Delta^4$; it was included here only for completeness.

For $\partial \Delta^5$: the numbers of triangles, edges and vertices are $N_2=20$, \ $N_1=15$, \ $N_0=6$, hence the dimension is $N_2-N_1+N_0-2 = 9$.

For the common boundary of the two sides of the Pachner move: $N_2=34$, \ $N_1=21$, \ $N_0=7$, hence $N_2-N_1+N_0-2 = 18$.
\end{proof}

\subsection{Explicit expressions}

Explicitly, triangle vectors can look as follows:
\begin{equation}\label{v12345}
  \begin{blockarray}{cccc}  
    \text{triangle }s & \BAmulticolumn{3}{c}{\text{components of\/ }\mathsf e_s \text{ on } u=12345}  \\[1ex]
    \begin{block}{c(ccc)}
123 & 0 & {\omega _{1345}} & {\omega _{1235}}-{\omega _{1234}} \quad \\
124 & 0 & 0 & -{\omega _{1235}}\\
125 & 0 & -{\omega _{1345}} & {\omega _{1234}}\\
134 & {\omega _{2345}} & -{\omega _{1235}} & {\omega _{1235}}\\
135 & -{\omega _{2345}} & {\omega _{1345}}+{\omega _{1235}} & -{\omega _{1234}}\\
145 & {\omega _{2345}} & -{\omega _{1235}} & 0\\
234 & -{\omega _{1345}} & 0 & -{\omega _{1235}}\\
235 & {\omega _{1345}} & -{\omega _{1345}} & {\omega _{1234}}\\
245 & -{\omega _{1345}} & 0 & 0\\
345 & { \quad \omega _{1345}}-{\omega _{2345}} & {\omega _{1235}} & 0 \\
    \end{block}
  \end{blockarray}
\end{equation}
Here, three components of~$\mathsf e_s|_{12345}$ are written out, corresponding to the fixed pentachoron $u=12345$ and all 10 triangles~$s\subset u$. For any other pentachoron $u=ijklm$, with $i<j<k<l<m$, one can just make in~\eqref{v12345} the substitution $1\mapsto i,\; \dots,\; 5\mapsto m$.

\section{Dual approach: triangle functionals and F-colorings}\label{s:f}

In Section~\ref{s:v}, we introduced ``V-colorings'', distinguished among all 3-component colorings of the 4-cells of a simplicial complex~$K$ by the fact that they belonged to the vector space spanned by ``triangle vectors''. Here, we introduce, dually, ``F-colorings'' that are distinguished by the fact that they belong to the nullspace of every ``triangle functional''.

\subsection{Definitions}\label{ss:fd}

Below, we consider simplicial complexes~$K$ representing a \emph{closed oriented} PL 4-manifold. The simplest example of these is the boundary~$\partial\Delta^5$ of one 5-cell, consisting, as we know, of six faces---pentachora~$\Delta^4$, often denoted here simply as~$u$. Recall that a ``coloring'' means, throughout this paper, that each 4-cell is assigned a ``color''---element of~$F^3$, where $F$ is a fixed field.

By definition, a \emph{triangle functional}~$\mathsf f_s$ is a linear form depending only on the colors on faces~$u\supset s$. Similarly to~\eqref{vo}, a triangle functional must change its sign together with the triangle orientation:
\begin{equation}\label{fo}
\mathsf f_{ijk} = -\mathsf f_{jik} = \mathsf f_{jki}\,, \quad \text{etc.}
\end{equation}

Given a set of triangle functionals for each triangle~$s$, we define \emph{F-colorings} of~$K$ as elements of the linear space on which all these functionals vanish.

Similarly to conditions $\mathrm{(i)}$--$\mathrm{(iii)}$ of Subsection~\ref{ss:vd}, we require then that our sets of triangle functionals have the following properties:
\begin{itemize}
 \item[$\mathrm{(i)}$] \emph{vanishing boundary}: for any given oriented edge~$b$, 
  \begin{equation}\label{fe}
   \sum_{s\supset b} \epsilon _{s,b} \mathsf f_s =0,
  \end{equation}
  where $\epsilon _{s,b}$ is the sign with which $b$ enters in the boundary~$\partial s$ of~$s$,
 \item[$\mathrm{(ii)}$] \emph{orthogonality to a given 2-cochain}:
  \begin{equation}\label{fw}
   \sum_{\mathrm{all\;}s} \gamma _s \mathsf f_s =0,
  \end{equation}
   where $\{\gamma _s\}$ are generic coefficients for all triangles~$s$, forming together a simplicial 2-\emph{cochain}~$\gamma$,
 \item[$\mathrm{(iii)}$] \emph{nondegeneracy}: for any pentachoron~$u$, the restrictions~$\mathsf f_s|_u$ of triangle functionals on~$u$ span the whole dual space to the color space~$F^3$.
\end{itemize}

\subsection{Existence and uniqueness up to gauge automorphisms}

\begin{theorem}\label{th:euf}
For a given~$K$ representing a closed oriented PL 4-manifold, there exists a set of triangle functionals~$\mathsf f_s$ satisfying the above conditions $\mathrm{(i)}$--$\mathrm{(iii)}$. Moreover, all such~$\mathsf f_s$ are determined uniquely up to gauge automorphisms of the color spaces on each separate 4-face.
\end{theorem}

\begin{proof}
We just repeat the proof of Theorem~\ref{th:euv}, with the two simple changes:
\begin{itemize}
 \item vector spaces are changed to their duals, that is, vectors to functionals,
 \item for the case of just one~$\Delta^4$, where we spoke about the 10 triangle vectors, we now speak of \emph{restrictions} of the 10 triangle functionals on the space of colors of that~$\Delta^4$. 
\end{itemize}
\end{proof}

\subsection{Triangle functionals and a simplicial 3-cocycle}

\begin{theorem}\label{th:g'}
Let complex~$K$ represent a triangulation of a closed oriented 4-dimen\-sional PL manifold having a trivial second cohomology~$H^2(K,F)=0$. Functionals~$\mathsf e_s$, satisfying conditions $\mathrm{(i)}$--$\mathrm{(iii)}$ of Subsection~\ref{ss:fd} and taken to within gauge automorphisms, depend only on the coboundary $\partial \gamma$ of cochain~$\gamma$.
\end{theorem}

\begin{proof}
Just repeat the proof of Theorem~\ref{th:b'} with obvious simple changes.
\end{proof}

\subsection{Dimensions of linear spaces of F-colorings}

\begin{theorem}\label{th:fd}
Let $K$ be as in Theorem~\ref{th:g'}, and recall that $N_i$ denotes the number of $i$-dimen\-sional simplices in it. Then, the dimension of the space of its F-colorings is
\begin{equation}\label{fM}
3N_4 - (N_2-N_1+N_0-2).
\end{equation}
\end{theorem}

\begin{proof}
A reasoning dual to the proof of Theorem~\ref{th:vd} shows that the total number of linearly independent linear \emph{restrictions} on colors is $N_2-N_1+N_0-2$. As the total number of color components on which these restrictions are imposed is~$3N_4$, we get~\eqref{fM}.
\end{proof}

Hence, we obtain the following dimensions of specific important spaces of F-colorings. Note that they are the same as in Theorem~\ref{th:3,9,18}!

\begin{theorem}\label{th:9,18}
There are:
 \begin{itemize}
  \item a 9-dimensional space of V-colorings for the boundary of one 5-simplex~$\Delta^5$,
  \item a 18-dimensional space of V-colorings for the common boundary of the l.h.s.\ and r.h.s.\ of the Pachner move 3--4.
 \end{itemize}
    \qed
\end{theorem}

\subsection{Explicit expressions and relation to V-colorings}

Here we present the explicit form for triangle functionals obeying conditions $\mathrm{(i)}$--$\mathrm{(iii)}$ of Subsection~\ref{ss:fd}---not just one of many possible forms, but such that satisfies the remarkable Theorem~\ref{th:vf} given here below.

We will use the same letter~$\omega$ for the coboundary $\omega=\partial\gamma$ of cochain~$\gamma$~\eqref{fw} as we used for~ $\partial\beta$ in \eqref{omega} and~\eqref{v12345}. This will be justified by the mentioned Theorem~\ref{th:vf}.

So, let $K$ be a triangulation of a closed oriented 4-dimen\-sional PL manifold. First, if it contains 4-simplex~$\Delta^4=12345$ and, moreover, the orientation of this~$\Delta^4$ in~$K$ coincides with its orientation induced by the order of its vertices, then the components of triangle functionals on this~$\Delta^4$ can be chosen as follows:
\begin{equation}\label{f12345}
  \begin{blockarray}{cccc}  
    \text{triangle }s & \BAmulticolumn{3}{c}{\text{components of\/ }\mathsf f_s \text{ on } u=12345}  \\[1ex]
    \begin{block}{c(ccc)}
123 & 0 & 0 & \dfrac{1}{{\omega _{1234}}\, {\omega _{1235}}}\\[2ex]
124 & \dfrac{1}{{\omega _{1245}}\, {\omega _{1345}}} & -\dfrac{{\omega _{1345}}-{\omega _{2345}}}{{\omega _{1235}}\, {\omega _{1245}}\, {\omega _{1345}}} & -\dfrac{{\omega _{1245}}+{\omega _{1234}}}{{\omega _{1234}}\, {\omega _{1235}}\, {\omega _{1245}}}\\[2ex]
125 & -\dfrac{1}{{\omega _{1245}}\, {\omega _{1345}}} & \dfrac{{\omega _{1345}}-{\omega _{2345}}}{{\omega _{1235}}\, {\omega _{1245}}\, {\omega _{1345}}} & \dfrac{1}{{\omega _{1235}}\, {\omega _{1245}}}\\[2ex]
134 & 0 & \dfrac{1}{{\omega _{1235}}\, {\omega _{1345}}} & \dfrac{1}{{\omega _{1234}}\, {\omega _{1235}}}\\[2ex]
135 & 0 & -\dfrac{1}{{\omega _{1235}}\, {\omega _{1345}}} & 0\\[2ex]
145 & \dfrac{1}{{\omega _{1245}}\, {\omega _{1345}}} & \dfrac{{\omega _{1235}}-{\omega _{1234}}}{{\omega _{1235}}\, {\omega _{1245}}\, {\omega _{1345}}} & -\dfrac{1}{{\omega _{1235}}\, {\omega _{1245}}}\\[2ex]
234 & -\dfrac{1}{{\omega _{1345}}\, {\omega _{2345}}} & -\dfrac{1}{{\omega _{1235}}\, {\omega _{1345}}} & -\dfrac{1}{{\omega _{1234}}\, {\omega _{1235}}}\\[2ex]
235 & \dfrac{1}{{\omega _{1345}}\, {\omega _{2345}}} & \dfrac{1}{{\omega _{1235}}\, {\omega _{1345}}} & 0\\[2ex]
245 & -\dfrac{{\omega _{2345}}+{\omega _{1245}}}{{\omega _{1245}}\, {\omega _{1345}}\, {\omega _{2345}}} & -\dfrac{{\omega _{1235}}-{\omega _{1234}}}{{\omega _{1235}}\, {\omega _{1245}}\, {\omega _{1345}}} & \dfrac{1}{{\omega _{1235}}\, {\omega _{1245}}}\\[2ex]
345 & \dfrac{1}{{\omega _{1345}}\, {\omega _{2345}}} & 0 & 0 \\[2ex]
    \end{block}
  \end{blockarray}
\end{equation}

\begin{remark}
We write thus both vectors in~\eqref{v12345} and covectors (functionals) in~\eqref{f12345} as \emph{rows}; we do it for the obvious typographical reasons.
\end{remark}

If $K$ contains a 4-simplex~$\Delta^4=ijklm$, with $i<j<k<l<m$, and the orientation of this~$\Delta^4$ in~$K$ coincides with its orientation induced by the order of its vertices, then substitution $1\mapsto i,\, \ldots,\, 5\mapsto m$ must be done in~\eqref{f12345}.

If, finally, the two mentioned orientations do \emph{not} coincide, the whole matrix in~\eqref{f12345} must be, additionally, multiplied by~$(-1)$.

\begin{theorem}\label{th:vf}
For the boundary~$K=\partial\Delta^5$ of one 5-simplex, the set of F-colorings defined using triangle functionals~\eqref{f12345} with such cochain~$\gamma$~\eqref{fw} that
\[
\partial \gamma = \omega
\] 
is the \emph{same} as the set of V-colorings defined using triangle vectors~\eqref{v12345}.
\end{theorem}

\begin{proof}
Direct calculation.
\end{proof}

\begin{remark}
Of course it would be interesting to find a more conceptual proof.
\end{remark}

\section{How heptagon follows from V- and F-colorings}\label{s:h}

\subsection{Five-dimensional manifolds with boundary and permitted colorings}

We call a coloring of the boundary of one 5-simplex \emph{permitted} if it is a V-, or, equivalently due to Theorem~\ref{th:vf}, F-coloring, constructed using a given generic 3-cocycle~$\omega$. We can also omit the word ``boundary'' and speak just of a permitted coloring of a 5-simplex. A \emph{cluster} of 5-simplices means below, generally, a triangulated 5-manifold with boundary, for instance, either l.h.s.\ or r.h.s.\ of Pachner move 3--4. By definition, a permitted coloring~$\mathfrak c$ of such a cluster~$M$---built again from a given simplicial 3-cocycle~$\omega$---is a coloring of all 4-faces whose restriction onto any 5-simplex is permitted, and a permitted coloring of the boundary~$\partial M$ is the restriction of~$\mathfrak c$ onto~$\partial M$.

\subsection{Inclusions for V-, permitted, and F-colorings}

\begin{theorem}\label{th:incl}
For the boundary~$\partial M$ of a cluster~$M$ of 5-simplices, denote the linear spaces of V-colorings, permitted colorings, and F-colorings, respectively, as $L_V$, $L_{\mathrm{perm}}$, and~$L_F$. Then, the following inclusions hold:
\begin{equation}\label{incl}
L_V \subseteq L_{\mathrm{perm}} \subseteq L_F
\end{equation}
\end{theorem}

\begin{proof}
To prove the first inclusion in~\eqref{incl} means to prove that any V-coloring of~$\partial M$ is permitted. Indeed, a V-coloring of~$\partial M$ is a linear combination of triangle vectors~$\mathsf v_s$ for triangles $s\subset \partial M$. Each~$\mathsf v_s$ can be extended from~$\partial M$ to~$M$ by adding the components belonging to inner 4-faces. Extending our triangle vectors this way, we get a V-coloring of each 5-simplex in~$M$ and hence a permitted coloring of $M$ and~$\partial M$.

To prove the second inclusion in~\eqref{incl} means to prove that any permitted coloring of~$\partial M$ is an F-coloring. Here we use the definition of a permitted coloring of~$\Delta^5$ in terms of F-colorings. Consider, for a triangle~$s\subset \partial M$, its \emph{star} in~$M$---the union of all 5-simplices~$v$ containing~$s$. For each~$v$, there is the triangle functional~$\mathsf f_s^{(v)}$, and their sum $\sum_v \mathsf f_s^{(v)} = \mathsf f_s^{(\partial M)}$ is nothing but the triangle functional for~$s$ in~$\partial M$, because all terms belonging to \emph{inner} 4-faces cancel out due to their different orientations with respect to the two adjacent 5-simplices, see the paragraphs between \eqref{f12345} and Theorem~\ref{th:vf}.
\end{proof}

\subsection{Heptagon}

Let now $M$ be either the l.h.s.\ or r.h.s.\ of Pachner move 3--4; the boundary~$\partial M$ is, as we know, the same. And its permitted colorings are also the same, because, combining~\eqref{incl} with the fact that
\[
\dim L_V = \dim L_F = 18
\]
due to Theorems \ref{th:3,9,18} and~\ref{th:9,18}, we see that simply
\[
L_V = L_{\mathrm{perm}} = L_F
\]
in this case. We have thus arrived at the following theorem.

\begin{theorem}\label{th:h}
Heptagon relation holds for permitted colorings of a 5-simplex defined as either V-colorings according to Section~\ref{s:v} or F-colorings according to Section~\ref{s:f}.
\qed
\end{theorem}

\end{document}